\documentclass[12pt]{amsart}
\usepackage{amsmath,amssymb,amsbsy,amsfonts,latexsym,amsopn,amstext,
                                               amsxtra,euscript,amscd,bm}
\usepackage[margin=1in]{geometry}                    
\usepackage{url}
\usepackage[colorlinks,linkcolor=blue,anchorcolor=blue,citecolor=blue]{hyperref}
\usepackage{color}
\usepackage{enumerate}

\begin{document}

\newtheorem{theorem}{Theorem}
\newtheorem{lemma}{Lemma}
\newtheorem{conjecture}{Conjecture}
\newtheorem{proposition}{Proposition}
\newtheorem{corollary}{Corollary}
\newtheorem{claim}{Claim}
\theoremstyle{definition}
\newtheorem{remark}{Remark}
\newtheorem{definition}{Definition}
\newtheorem{algorithm}{Algorithm}

\def\E{{\mathbb E}}
\def\F{{\mathbb F}}
\def\R{{\mathbb R}}
\def\Z{{\mathbb Z}}
\def\C{{\mathbb C}}
\def\N{{\mathbb N}}


\title[Bilinear Kloosterman sums and uniformity of a random walk]{Bilinear Kloosterman Sums over Small Boxes and uniformity of a random walk}

\author{Ali Mohammadi}

\email{ali.mohammadi.np@gmail.com}


\pagenumbering{arabic}

\begin{abstract}
Given an additive character $\psi$ of an arbitrary finite field $\F_{p^n}$ and elements $a,b\in \F_{p^n}$, $b\neq 0$, we prove a bound on bilinear Kloosterman sums $\sum_{x\in B_1}\sum_{y\in B_2}\psi(axy+bx^{-1}y^{-1})$, where $B_1, B_2$ denote boxes in $\F_{p^n}$. Here, a box is a coordinate parallelepiped obtained by restricting the coefficients of field elements, with respect to a fixed basis of $\F_{p^n}$ over $\F_p$, to intervals in $\F_p$. Our estimates are nontrivial when $|B_1||B_2|>p^{n/2+\epsilon}$ and hence in a range not accessible by the Weil bound.

 We also consider the random walk on $\F_{p^n}$ defined by $S_k=S_0+W_1+\cdots+W_k$,
where $W_i=aX_iY_i+b(X_iY_i)^{-1}$ and $X_i,Y_i$ are independent uniformly distributed random variables on $B_1$ and $B_2$ respectively. We show that the nontrivial Fourier coefficients of $S_k$ decay exponentially. Consequently, every nonzero $\F_p$-linear projection of $S_k$, as well as the full distribution of $S_k$ converge to the uniform distribution on $\F_p$ and $\F_{p^n}$, respectively. We also obtain bounds on the rate at which the entropy of $S_k$ converges to its maximal value.
\end{abstract}

\maketitle

\section{Introduction}
Let $p$ denote a prime, $n$ a natural number and $q = p^n$. We use $\F_q$ to denote the finite field of $q$ elements and $\F_q^*$ as the multiplicative group $\F_q\setminus\{0\}.$  We write $e_p(x) = exp(2\pi i x/p)$ and $\psi(x) = e_p(Tr(x))$, where $Tr(x) = x+x^{p} + \cdots + x^{p^{n-1}}$ denotes the trace of $x\in \F_q$ over $\F_p$.

In \cite[Theorem~A.1]{Bou}, Bourgain proved that, given an arbitrary $\epsilon>0$ and intervals $I, J \subset \left[1, p\right]$ with $|I||J|>p^{1/2+\epsilon}$, one has
\begin{equation}
    \label{eqn:BourgainBKS}
    \max_{(a,b)\in \F_p\times\F_p^*}  \bigg|\sum_{x\in I}\sum_{y\in J}e_p(axy+bx^{-1}y^{-1})\bigg|\ll_{\epsilon}p^{-\delta}|I||J|,
\end{equation}
for some $\delta = \delta(\epsilon)>0$.

We extend Bourgain's theorem to coordinate boxes in arbitrary finite fields. Given an arbitrary basis $\Omega=\{\omega_1, \dots, \omega_n\}$ of $\F_{q}$ over $\F_p$, each $x\in \F_q$ has the unique representation
\begin{equation*}
    x = x_1\omega_1 + \cdots + x_n\omega_n,\quad \text{with} \quad 0\leq x_i <p.
\end{equation*}
For $1\leq H \leq p$ and
\begin{equation}
\label{eqn:Ndef}
   \bm{N} = (N_1, \dots, N_n)\quad\text{with} \quad 0\leq N_i < N_i + H \leq p, \quad 1\leq i \leq n,
\end{equation}
we define a box $B(\bm{N}, H)$ as the set
\begin{equation}
    \label{def:Box}
    \bigg\{\sum_{i=1}^n x_i\omega_i: x_i\in \left[N_i+1, N_i +H\right], 1\leq i\leq n\bigg\}.
\end{equation}
Geometrically, such sets are the natural higher-dimensional analogue of intervals, and may be viewed as parallelepipeds in the coordinate system determined by $\Omega$. Our main theorem establishes a power-saving estimate for weighted bilinear Kloosterman sums over arbitrary boxes whenever $|B_1||B_2|\ge q^{1/2+\varepsilon}$, thereby extending Bourgain's result from $\F_p$ to $\F_q$.

Our Kloosterman sums estimate
implies that the nonlinear transformation $(x,y)\longmapsto axy+b(xy)^{-1}$ destroys this additive structure in a Fourier-analytic sense. Let
$X_i$ and $Y_i$ be independent random variables uniformly distributed
on two sufficiently large boxes $B_1,B_2\subset\F_q^*$, and define $W_i=aX_iY_i+b(X_iY_i)^{-1}$. We consider the random walk $S_k=S_0+W_1+\cdots+W_k$, where $S_0$ is an arbitrary $\F_q$-valued random variable. We prove that repeated independent increments progressively destroy the coordinate structure inherited from the boxes, driving the distribution towards uniformity and the entropy towards its maximum.

\subsection*{Notation}
For real-valued sequences $\alpha = \{\alpha_n\}_{n=1}^{\infty}$ and $\beta = \{\beta_n\}_{n=1}^{\infty}$, we write $\alpha\ll \beta$, $\beta\gg \alpha$, $\alpha = O(\beta)$ or $\beta = \Omega(\alpha)$ if there exists an $n_0\geq 1$ and a constant $c>0$ such that $\alpha_n\leq c\beta_n$ for all $n\geq n_0$. If the constant $c$, depends on some parameter $\epsilon$, then we write, for example,  $\alpha\ll_{\epsilon} \beta$. If $\alpha\ll \beta$ and $\alpha\gg \beta$, then we use $\alpha \approx \beta$. We also write $\alpha \lesssim \beta$ or $\beta \gtrsim \alpha$, if there exist $c_1, c_2>0$ and $n_0\geq 1$ such that $\alpha_n \leq c_1 (\log \beta_n)^{c_2}\beta_n$ for all $n\geq n_0$.

For a random variable $Z$ taking values in a finite set, we write $\mathcal L(Z)$ for its probability law; that is, $\mathcal L(Z)(x)$ represents $\Pr(Z=x)$.  We also write $U_p$ and $U_q$ for the uniform probability distribution on $\F_p$ and $\F_q$, respectively.

\section{Main results}
\subsection{Bilinear Kloosterman sum estimate}
We prove the following $\F_q$ variant of Bourgain's estimate~\eqref{eqn:BourgainBKS}.
\begin{theorem}
\label{thm:BKSOB}
Let $B_1, B_2\subset \F_q$ denote arbitrary boxes of the form \eqref{def:Box} satisfying 
\begin{equation}
    \label{eqn:B1B2LBC}
|B_1||B_2|\geq q^{1/2+\epsilon}
\end{equation}
for some $\epsilon>0$. Let $\alpha$ and $\beta$ be complex-valued functions satisfying $|\alpha(x)|, |\beta(x)| \leq 1$ for all $x\in \F_q$. Then, there exists $\delta = \delta(\epsilon)>0$ such that
\begin{equation*}
  \max_{(a,b)\in \F_q\times\F_q^*}  \bigg|\sum_{x\in B_1}\sum_{y\in B_2}\alpha(x)\beta(y)\psi(axy+bx^{-1}y^{-1})\bigg|\ll_{\epsilon}p^{-\delta}|B_1||B_2|.
\end{equation*}
\end{theorem}
As a consequence, we obtain the following result on expansion of sum sets of reciprocals over boxes.
\begin{corollary}
\label{thm:kfssrb}
Let $B$ denote any box of the form \eqref{def:Box} and let 
$$
kB^{-1} = \bigg\{\frac{1}{x_1}+ \cdots +\frac{1}{x_k}: (x_1, \dots, x_k)\in B^k\bigg\}.
$$
Then, given any $\delta>0$, there exists $k\in \N$ such that
\begin{equation*}
    |kB^{-1}| \geq p^{-\delta}\min \{|B|^{2}, q\}.
\end{equation*}
\end{corollary}

\subsection{Uniformity bounds for a random walk}
For probability distributions $P$ and $Q$ on a finite set, define
the relative entropy
\[
D(P\Vert Q)
=
\sum_xP(x)\log\frac{P(x)}{Q(x)},
\]
the chi-square divergence
\[
\chi^2(P\Vert Q)
=
\sum_x\frac{(P(x)-Q(x))^2}{Q(x)},
\]
and the total variation distance
\[
\|P-Q\|_{\mathrm{TV}}
=
\frac12\sum_x|P(x)-Q(x)|.
\]

We write $\pi_j(z)=z_j$
for the $j$-th coordinate map. For an $\F_q$-valued random variable $Z$, define its coordinate
entropy with respect to the basis $\Omega$ by
\[
H_\Omega(Z)
=
\sum_{j=1}^n H(\pi_j(Z)),
\]
where
\[
H(\pi_j(Z))
=
-\sum_{t\in\F_p}
\Pr(\pi_j(Z)=t)
\log\Pr(\pi_j(Z)=t).
\]
Since each coordinate takes values in $\F_p$, $H_\Omega(Z)\le n\log p$. We study how rapidly the random walk $S_k=S_0+W_1+\cdots+W_k$ loses the coordinate-restrained structure inherited from the boxes. We quantify
this mixing in several complementary ways. We measure the
distance between the distribution of the walk and the uniform
distribution using the chi-square divergence and total
variation distance and also prove lower bounds on its entropy.

The next theorem shows that every nonzero $\F_p$-linear observation of
the walk becomes exponentially close to uniform. In particular, every
coordinate in every basis of $\F_q$ approaches the uniform distribution
on $\F_p$, while the total coordinate entropy approaches its maximal value.

\begin{theorem}
\label{thm:digit-mixing}
Let $B_1,B_2\subset\F_q^*$ be arbitrary boxes of the form \eqref{def:Box} satisfying $|B_1||B_2|\ge q^{1/2+\epsilon}$for some $\varepsilon>0$. For $a\in\F_q, b\in\F_q^*$, let $W_i=aX_iY_i+b(X_iY_i)^{-1}$, where the pairs $(X_i,Y_i)$ are independent, $X_i$ is uniform on
$B_1$, and $Y_i$ is uniform on $B_2$. Let $S_0$ be an arbitrary $\F_q$-valued random variable independent
of the increments, and define $S_k=S_0+W_1+\cdots+W_k$. Set $\rho=\min\{1,C_\epsilon p^{-\delta}\}$. Then, for every nonzero $\F_p$-linear map $L:\F_q\longrightarrow\F_p$, one has
\begin{equation}
\label{eq:projection-chi-square}
\chi^2\bigl(\mathcal L(L(S_k))\Vert U_p\bigr)
\le
(p-1)\rho^{2k},
\end{equation}
and hence
\begin{equation}
\label{eq:projection-TV}
\left\|
\mathcal L(L(S_k))-U_p
\right\|_{\mathrm{TV}}
\le
\frac12\sqrt{p-1}\,\rho^k.
\end{equation}
Moreover,
\begin{equation}
\label{eq:projection-entropy}
H(L(S_k))
\ge
\log p-
\log\left(1+(p-1)\rho^{2k}\right).
\end{equation}
In particular, for every basis
$\Omega=\{\omega_1,\dots,\omega_n\}$,
\begin{equation}
\label{eq:coordinate-entropy}
H_\Omega(S_k)
\ge
n\log p-
n\log\left(1+(p-1)\rho^{2k}\right).
\end{equation}
\end{theorem}



While Theorem~\ref{thm:digit-mixing} controls every one-dimensional
linear projection of the walk, the next theorem extends this to the
entire $\F_q$-valued distribution, showing that the random walk itself
becomes asymptotically uniform on~$\F_q$.

\begin{theorem}
\label{thm:full-field-mixing}
Under the assumptions of Theorem~\ref{thm:digit-mixing},
\begin{equation}
\label{eq:full-chi-square}
\chi^2\bigl(\mathcal L(S_k)\Vert U_q\bigr)
\le
(q-1)\rho^{2k}.
\end{equation}
Consequently,
\begin{equation}
\label{eq:full-TV}
\left\|
\mathcal L(S_k)-U_q
\right\|_{\mathrm{TV}}
\le
\frac12\sqrt{q-1}\,\rho^k,
\end{equation}
and
\begin{equation}
\label{eq:full-entropy}
H(S_k)
\ge
\log q-
\log\left(1+(q-1)\rho^{2k}\right).
\end{equation}
\end{theorem}

\begin{remark}
Entropy is monotone along the
walk. Indeed, if $U$ and $V$ are independent random variables on a
finite abelian group, then $H(U+V)\ge H(U)$ and
$H(U+V)\ge H(V)$, since $H(U+V\mid V)=H(U)$ by translation invariance of entropy, while conditioning cannot increase
entropy. Applying this with $S_{k+1}=S_k+W_{k+1}$,
where $S_k$ and $W_{k+1}$ are independent, yields $H(S_{k+1})\ge H(S_k)$. The same argument applied to every nonzero $\F_p$-linear map
$L$ gives $H(L(S_{k+1}))\ge H(L(S_k))$,
and hence, in particular, $H_\Omega(S_{k+1})\ge H_\Omega(S_k)$.
\end{remark}

\begin{remark}
Theorems~\ref{thm:digit-mixing} and
\ref{thm:full-field-mixing} exhibit two distinct mixing scales. Since $\rho\le C_\epsilon p^{-\delta}$,
equations~\eqref{eq:projection-TV} and~\eqref{eq:full-TV} imply
\[
\left\|
\mathcal L(L(S_k))-U_p
\right\|_{\mathrm{TV}}
\le
\frac12 C_\epsilon^k
p^{1/2-k\delta}
\]
for every nonzero $\F_p$-linear map $L$, and
\[
\left\|
\mathcal L(S_k)-U_q
\right\|_{\mathrm{TV}}
\le
\frac12 C_\epsilon^k
p^{n/2-k\delta}.
\]

Consequently, for any fixed $k>1/(2\delta)$,
every nonzero linear projection of $S_k$ is asymptotically uniform on
$\F_p$ as $p\to\infty$ (i.e. the norm becomes of order $o_{p\to\infty}(1)$). In particular, every coordinate with respect
to every basis of $\F_q$ over $\F_p$ becomes nearly uniform after
only $O(1)$ steps. By contrast, convergence of the full $\F_q$-valued distribution
requires $k>n/(2\delta)$,
reflecting the need to control all $q-1$ nontrivial additive
characters simultaneously. This simply highlights the fact that coordinate-wise pseudorandomness
occurs before full joint pseudorandomness, with the latter requiring an
additional factor of $n$.
\end{remark}

\section{Preparations}
First, we recall some basic facts related to the theory of Fourier analysis over finite abelian groups, which we state for the special case of $\F_q$. The reader may consult \cite[Chapter~4]{TaoVu} for more details. Writing $\psi_a(x) = \psi(ax)$, the set $\{\psi_a:a\in \F_q\}$ represents the set of characters of $\F_q$. Then, every function $f:\F_q\rightarrow \C$ may be written as 
\begin{equation}
\label{eqn:frepch}
    f = \sum_{a\in \F_q}\hat{f}(a)\cdot \psi_a,
\end{equation}
where 
\begin{equation*}
    \hat{f}(a) = \E_{x\in\F_q}f(x)\overline{\psi_a(x)} \quad \bigg(\text{writing}\quad \E_{x\in A} g(x) := \frac{1}{|A|}\sum_{x\in A} g(x)\bigg).
\end{equation*}

For $f, g:\F_q\rightarrow \C$, we recall the Parseval identity
\begin{equation}
\label{eqn:Plancherel}
    \E_{x\in \F_q} f(x)\overline{g(x)} = \sum_{a\in \F_q}\hat{f}(a)\overline{\hat{g}(a)}.
\end{equation}
We define the convolution of $f$ and $g$ by
\begin{equation*}
    f*g(x) = \E_{y\in \F_q}f(y)g(x-y)
\end{equation*}
and write $f^{(k)}$ for the $k$-fold convolution of $f$. Then, for any $a\in \F_q$, we have
\begin{equation}
\label{eqn:ConvFT}
    \widehat{f*g}(a) = \hat{f}(a)\cdot\hat{g}(a).
\end{equation}

Applying \eqref{eqn:Plancherel} and \eqref{eqn:ConvFT}, it follows that, for arbitrary $A\subset\F_q$ and $x\in \F_q$, we have
\begin{equation}
    \label{eqn:RepFConv}
    1_A^{(k)}(x) = \frac{1}{q^{k-1}}\bigg|\bigg\{(a_1,\dots,a_k)\in A^k:a_1+\cdots+a_k=x\bigg\}\bigg|,
\end{equation}
where $1_A(x)$ denotes the indicator function of $A$.

We require the following standard pigeonholing argument.
\begin{lemma}
\label{lem:PopPig}
Let $0<\mu < 1$, $X\subset \F_q$, $f$ a function with $f(x)\geq 0$ for all $x\in X$ and suppose
$$
\sum_{x\in X}f(x)\geq K.
$$
Let $Y = \{x\in X: f(x)\geq \mu K/|S|\}$. Then
$$
\sum_{y\in Y}f(y) \geq (1-\mu)K.
$$
If we further assume that $f(x)\leq M$ for all $x\in X$, then $|Y|\geq (1-\mu)K/M.$
\end{lemma}
The following result appears in \cite[Page 664]{BouGar2}. Also see \cite[Lemmas 2 and 3]{Bak}.
\begin{lemma}
\label{lem:recSol}
Let $X, Y\subset \F_q$. Then
\begin{align*}
   \bigg|\bigg\{(x_1,\dots,x_{2n}, y)&\in X^{2n}\times Y: \frac{1}{y+x_1}+\cdots +\frac{1}{y+x_n} = \frac{1}{y+x_{n+1}}+\cdots +\frac{1}{y+x_{2n}}\bigg\}\bigg|
   \\ &\ll |X|^n|Y| + |X|^{2n}.
\end{align*}
\end{lemma}

Given sets $X, Y\subset\F_q$, we define the multiplicative energy between $X$ and $Y$ by
$$
E_{\times}(X, Y) = |\{(x_1, x_2, y_1, y_2)\in X^2\times Y^2:x_1 y_1 = x_2 y_2\}|,
$$
and write $E_{\times}(X) = E_{\times}(X, X)$. We also recall the well-known inequality (see \cite[Corollary~2.10]{TaoVu})
\begin{equation}
    \label{eqn:ECS}
    E_{\times}(X, Y) \leq E_{\times}(X)^{1/2}E_{\times}(Y)^{1/2}.
\end{equation}
We require the following bound on the multiplicative energy of boxes, proved in \cite[Lemma~1]{Kon}.
\begin{lemma}
\label{lem:KonMEB}
Given $H\leq p^{1/2}$ and an arbitrary $\bm{N}$ of the form \eqref{eqn:Ndef}, let $B = B(\bm{N}, H)$ denote the set defined by \eqref{def:Box}. Then $E_{\times}(B) \ll_n \log{p}\cdot|B|^2$.
\end{lemma}
We also need the following result on the sub-additivity of the multiplicative energy (see for example \cite[Lemma~8]{KonShk}).
\begin{lemma}
\label{lem:ESubadd}
Given sets $X_1, \dots, X_k \subseteq \F_q$, we have
\begin{equation*}
\label{eqn:Esubadd1}
    E_{\times}\bigg(\bigcup_{i=1}^{k}X_i\bigg) \leq \bigg(\sum_{i=1}^k E_{\times}(X_i)^{1/4}\bigg)^4.
\end{equation*}
\end{lemma}

\begin{lemma}
\label{lem:wrapped_box_energy}
Let $H\le p^{1/2}$ and, for $1\le i\le n$, let $I_i\subseteq \F_p$ be an interval of length $H$, where the interval is allowed to wrap around modulo $p$. Define
\[
B=
\left\{
\sum_{i=1}^n x_i\omega_i:
x_i\in I_i,\ 1\le i\le n
\right\}.
\]
Then $E_\times(B)\ll_n \log p\cdot|B|^2$.
\end{lemma}

\begin{proof}
For each $i$, if $I_i$ does not wrap around modulo $p$, then it is an interval of the form $[N_i+1,N_i+H]$ for some $0\le N_i<N_i+H\le p$. Otherwise,
\[
I_i=[1,a_i]\cup[p-b_i+1,p],
\]
where $a_i+b_i=H$. Thus every wrapped interval is the disjoint union of two ordinary intervals. Consequently, $B$ can be written as the disjoint union of at most $2^n$ boxes
\[
B=\bigcup_{j=1}^{m}B_j,
\qquad m\le 2^n,
\]
where each $B_j$ is a box of the form \eqref{def:Box}. By Lemma~\ref{lem:KonMEB}, $E_\times(B_j)\ll_n \log p\cdot|B_j|^2$ for every $j$. Applying Lemma~\ref{lem:ESubadd} and H\"older's inequality,
\[
E_\times(B)
\le
\left(
\sum_{j=1}^{m}
E_\times(B_j)^{1/4}
\right)^4
\ll_n
\log p\cdot
\left(
\sum_{j=1}^{m}
|B_j|^{1/2}
\right)^4 \ll_n \log p\cdot |B|^2,
\]
as required.
\end{proof}

\section{Proofs of Theorem~\ref{thm:BKSOB} and Corollary~\ref{thm:kfssrb}}
We follow a similar scheme as \cite{Bou} to break down the proofs into a number of auxiliary results as follows.
\begin{lemma}
\label{lem:LSpecUB}
For $H\leq p^{1/2}$ and arbitrary $\bm{N}$ of the form \eqref{eqn:Ndef}, let $B = B(\bm{N}, H)$ be the set given by \eqref{def:Box}. Then, given $\epsilon >0$, writing 
$$
\Omega =  \left\{\xi\in \F_q:\left|\sum_{x\in B}\psi\left(\xi/ x\right)\right|>p^{-\epsilon}|B|\right\},
$$
we have
\begin{equation*}
    |\Omega| \lesssim_n q^{1+\delta} |B|^{-2}.
\end{equation*}
for some $\delta = \delta(\epsilon)>0$, with $\delta\approx \epsilon^{1/2}$.
\end{lemma}
\begin{proof}
Let $\tau>0$ be arbitrary, $a\in I = (0, p^{\tau})$, $b\in B_* = B(\bm{N}, p^{-2\tau}H)$ and note that
\begin{align*}
    \bigg| \sum_{x\in B}\psi\bigg(\frac{\xi}{x}\bigg) - \sum_{x\in B}\psi\bigg(\frac{\xi}{x+ab}\bigg)\bigg|  &\leq |B\setminus (B+ab)| + |(B+ab)\setminus B| \\ &< 2np^{-\tau}H^n.
\end{align*}
Assuming that $\tau>\epsilon$, we deduce
\begin{align}
\label{eqn:trsrep}
\nonumber   |\Omega||B||B_*||I|p^{-\epsilon}&\ll_n \sum_{\xi\in \Omega}\bigg|\sum_{x\in B}\sum_{a\in B_*}\sum_{b\in I}\psi\bigg(\frac{\xi}{x+ab}\bigg)\bigg|
\nonumber   \\ &\leq \sum_{\xi\in \Omega}\sum_{x\in B}\sum_{a\in B_*}\bigg|\sum_{b\in I}\psi\bigg(\frac{a^{-1}\xi}{a^{-1}x+b}\bigg)\bigg| \\
\nonumber   &= \sum_{\lambda_1, \lambda_2\in \F_q}r(\lambda_1, \lambda_2)\bigg|\sum_{b\in I}\psi\bigg(\frac{\lambda_1}{\lambda_2 + b}\bigg)\bigg|,
\end{align}
where
\begin{equation*}
    r(\lambda_1, \lambda_2) = \bigg|\bigg\{(x, a, \xi)\in B\times B_* \times \Omega: \bigg(\frac{\xi}{a}, \frac{x}{a}\bigg) = (\lambda_1, \lambda_2)\bigg\}\bigg|.
\end{equation*}
Note that
\begin{equation}
\label{eqn:rep1stmmnt}
    \sum_{\lambda_1, \lambda_2}r(\lambda_1, \lambda_2) = |\Omega||B||B_*|.
\end{equation}
We also have
\begin{equation*}
    \sum_{\lambda_1, \lambda_2}r(\lambda_1, \lambda_2)^2 = \bigg|\bigg\{(x_1, x_2, a_1, a_2, \xi_1, \xi_2)\in B^2\times B_*^2\times \Omega^2: \bigg(\frac{\xi_1}{a_1},  \frac{x_1}{a_1}\bigg) =\bigg(\frac{\xi_2}{a_2},\frac{x_2}{a_2}\bigg) \bigg\}\bigg|
\end{equation*}
such that, by \eqref{eqn:ECS} and Lemma~\ref{lem:KonMEB},
\begin{align}
\label{eqn:rep2ndmmnt}
    \sum_{\lambda_1, \lambda_2}r(\lambda_1, \lambda_2)^2 &\ll |\Omega|E_{\times}(B, B_*)\\ \nonumber
    &\leq |\Omega|E_{\times}(B)^{1/2}E_{\times}(B_*)^{1/2} \\ \nonumber
   &\ll_n \log{p} \cdot |\Omega||B||B_*|.
\end{align}
By a double application of H\"older's inequality, then using \eqref{eqn:rep1stmmnt} and \eqref{eqn:rep2ndmmnt}, we bound \eqref{eqn:trsrep} by
\begin{align*}
        &\leq \bigg(\sum_{\lambda_1, \lambda_2}r(\lambda_1, \lambda_2)\bigg)^{1-\frac{1}{k}}\bigg(\sum_{\lambda_1, \lambda_2}r(\lambda_1, \lambda_2)^2\bigg)^{\frac{1}{2k}}\bigg(\sum_{\lambda_1, \lambda_2}\bigg|\sum_{b\in I}\psi\bigg(\frac{\lambda_1}{\lambda_2+b}\bigg)\bigg|^{2k}\bigg)^{\frac{1}{2k}}\\
    &\lesssim (|\Omega||B||B_*|)^{1-\frac{1}{2k}}\cdot(*)^{\frac{1}{2k}},
\end{align*}
where
\begin{align*}
   (*) &=  \sum_{\lambda_1, \lambda_2\in \F_q}\sum_{b_1, \dots, b_{2k}\in I}\psi\bigg(\lambda_1\bigg(\frac{1}{\lambda_2+b_1}+ \cdots - \frac{1}{\lambda_2+b_{2k}}\bigg)\bigg)\\
   &=q\bigg|\bigg\{(b_1, \dots, b_{2k}, \lambda_2)\in I^{2k}\times \F_q: \frac{1}{\lambda_2+b_1}+ \cdots - \frac{1}{\lambda_2+b_{2k}} = 0\bigg\}\bigg|\\
   &\ll q(q\cdot p^{\tau k} + p^{2\tau k}),
\end{align*}
having used the orthogonality of characters and Lemma~\ref{lem:recSol} to deduce the last two lines respectively. We assume $k> n\tau^{-1}$,
so that
\begin{equation}\label{eqn:starUB}
    (*)\ll q\cdot p^{2\tau k}.
\end{equation}
Thus, returning to \eqref{eqn:trsrep}, we have
\begin{align*}
    |\Omega||B||B_*||I|p^{-\epsilon}&\lesssim_n (|\Omega||B||B_*|)^{1-\frac{1}{2k}}\cdot(q\cdot p^{2\tau k})^{\frac{1}{2k}}.
\end{align*}
Using that $|B_*| = p^{-2\tau n}|B|$, we deduce
\begin{align*}
    |\Omega| \lesssim_n |B|^{-2}q\cdot p^{2\epsilon k + 2\tau n}.
\end{align*}
We choose $k = 2n\lceil\tau^{-1}\rceil$ and $\tau = \epsilon^{1/2}$ to conclude
\begin{equation*}
    |\Omega| \lesssim_n |B|^{-2}q^{1+10 \sqrt{\epsilon}}
\end{equation*}
as required.
\end{proof}
\begin{lemma}
\label{lem:LSpectoSConv}
For $A\subset\F_q$ and $\epsilon >0$, let 
$$
\Lambda_{\epsilon} = \bigg\{\xi\in \F_q:\bigg|\sum_{x\in A}\psi(\xi x)\bigg|>p^{-\epsilon}|A|\bigg\}.
$$
Given $\tau>0$, $k>(1+n/2)\epsilon^{-1}$ and $S\subset\F_q$, if $|\Lambda_{\epsilon}||S|<p^{-\tau}q$, then
$$
M(A) := \bigg|\bigg\{(a_1, \dots, a_k)\in A^k:a_1+\cdots+a_k\in S\bigg\}\bigg|<p^{-\tau}|A|^{k}.
$$
\end{lemma}
\begin{proof}
Assume that $M(A) \geq p^{-\tau}|A|^k$ and let $\nu(x) = 1_{A}(x)/|A|$. By \eqref{eqn:RepFConv}, we have
$$
M(A) = |A|^{k}q^{k-1}\sum_{x\in S}\nu^{(k)}(x)\quad \text{and so}\quad \sum_{x\in S}\nu^{(k)}(x) \geq q^{1-k}p^{-\tau}.
$$
Recalling \eqref{eqn:frepch}, \eqref{eqn:ConvFT} and \eqref{eqn:Plancherel}, we have
\begin{align*}
\sum_{x\in S}\nu^{(k)}(x) &= \sum_{x\in S}\sum_{a}\hat{\nu}^k(a)\psi_a(x) \\
&=\sum_{a}\hat{\nu}^k(a)\sum_{x\in S}\psi_a(x)\\
&=q\sum_{a}\hat{\nu}^k(a)\overline{\hat{1}_{S}(a)}.
\end{align*}
Hence
\begin{equation}
\label{eqn:kpsumLB}
\sum_{a}\hat{\nu}^k(a)\overline{\hat{1}_{S}(a)} \geq q^{-k}p^{-\tau}.
\end{equation}
Now, note that
\begin{align*}
    \bigg|\sum_{a\not\in \Lambda_{\epsilon}}\hat{\nu}^k(a)\overline{\hat{1}_{S}(a)}\bigg|&\leq\sum_{a\not\in \Lambda_{\epsilon}}\big|\hat{\nu}^k(a)\big|\big|\hat{1}_{S}(a)\big|\\
    &\leq \max_{a\not\in \Lambda_{\epsilon}}\big|\hat{\nu}^k(a)\big|\cdot \sum_{a\in \F_q}\big|\hat{1}_{S}(a)\big|.
\end{align*}
Furthermore, for $a\not\in\Lambda_{\epsilon}$, using the definition of $\Lambda_{\epsilon}$, we get
$$
|\hat{\nu}(a)|= \frac{1}{q|A|}\bigg|\sum_{x\in A}\overline{\psi_a}(x)\bigg|<q^{-1}p^{-\epsilon},
$$
which implies 
$$
\max_{a\not\in \Lambda_{\epsilon}}\big|\hat{\nu}^k(a)\big| < q^{-k}p^{-k\epsilon}.
$$
Next, by the Cauchy-Schwarz inequality and \eqref{eqn:Plancherel}, we have
\begin{align*}
    \sum_{a}\big|\hat{1}_S(a)\big| &\leq q^{1/2}\bigg(\sum_{a}\big|\hat{1}_S(a)\big|^2\bigg)^{1/2}\\
    &=|S|^{1/2}.
\end{align*}
Hence, using the assumption that $k> (1+n/2)\epsilon^{-1}$, we have
\begin{align*}
\bigg|\sum_{a\not\in \Lambda_{\epsilon}}\hat{\nu}^k(a)\overline{\hat{1}_{S}(a)}\bigg| &\leq q^{-k} p^{-k\epsilon}\cdot |S|^{1/2} \\
&=o(q^{-k}p^{-\tau}).
\end{align*}
Therefore, returning to~\eqref{eqn:kpsumLB}, we have
\begin{equation}
\label{eqn:rhoLambdSLB}
    \sum_{a\in \Lambda_{\epsilon}}\hat{\nu}^k(a)\overline{\hat{1}_{S}(a)} > q^{-k}p^{-\tau}.
\end{equation}
By \eqref{eqn:Plancherel}, we may rewrite the left-hand side of \eqref{eqn:rhoLambdSLB} as
\begin{equation}
\sum_{a}\widehat{\nu^{(k)}}(a) \cdot \overline{\widehat{(1_{S}* \check{1}_{\Lambda_{\epsilon}})}(a)} = \E_{x\in \F_q}\nu^{(k)}(x)\cdot \overline{1_{S} * \check{1}_{\Lambda_{\epsilon}}(x)}.
\end{equation}
Recalling \eqref{eqn:RepFConv}, we have
$$
\sum_{x\in \F_q}\nu^{(k)}(x) =q^{1-k},
$$
implying
\begin{equation*}
    \E_{x\in \F_q}\nu^{(k)}(x)\cdot \overline{1_{S} * \check{1}_{\Lambda_{\epsilon}}(x)} \leq\frac{1}{q^{k}}\cdot \|1_S * \check{1}_{\Lambda_{\epsilon}}\|_{\infty}.
\end{equation*}
Going back to \eqref{eqn:rhoLambdSLB}, we deduce
\begin{equation*}
    \|1_S * \check{1}_{\Lambda_{\epsilon}}\|_{\infty} \geq p^{-\tau}.
\end{equation*}
By Young's inequality (see \cite[Equation~4.13]{TaoVu}), we have
$$
 \|1_S * \check{1}_{\Lambda_{\epsilon}}\|_{\infty} \leq \|1_S\|_{L^{1}}\cdot \|\check{1}_{\Lambda_{\epsilon}}\|_{\infty},
$$
where
$$
\|1_S\|_{L^{1}} = \frac{1}{q}\sum_{x\in \F_q}|1_S(x)| = \frac{|S|}{q},
$$
and
$$
\|\check{1}_{\Lambda_{\epsilon}}\|_{\infty} = \max_{x\in \F_q}\Big|\sum_{a} 1_{\Lambda_{\epsilon}}(a)\psi_a(x)\Big| \leq  |\Lambda_{\epsilon}|.
$$
Therefore $|\Lambda_{\epsilon}||S| > q\cdot p^{-\tau}
$ as required.
\end{proof}
\begin{corollary}
\label{cor:convconc}
For $H\leq p^{1/2}$ and arbitrary $\bm{N}$ of the form \eqref{eqn:Ndef} let $B = B(\bm{N}, H)$ be the set given by \eqref{def:Box}. Suppose that for some arbitrary $\delta>0$ and $S\subset \F_q$, we have $|S|<p^{-2\delta}|B|^2$. Then there exists $k=k(n,\delta)\in\N$ such that
\begin{equation}
    M(A) = \bigg|\bigg\{(a_1, \dots, a_k)\in A^k:a_1+\cdots+a_k\in S\bigg\}\bigg|<p^{-\delta}|A|^{k}.
\end{equation}
\end{corollary}
\begin{proof}
We use Lemma~\ref{lem:LSpectoSConv} with $A = \{x^{-1}: x\in B\}.$ For arbitrary $\delta>0$, with an appropriate choice of $\epsilon >0$,  Lemma~\ref{lem:LSpecUB} ensures $|\Lambda_\epsilon|<p^{\delta}q|B|^2$. Therefore, taking $\tau = \delta$ in Lemma~\ref{lem:LSpectoSConv} gives the desired result.
\end{proof}
\begin{lemma}
\label{lem:LSpec2}
For $H\leq p^{1/2}$ and arbitrary $\bm{N}$ of the form \eqref{eqn:Ndef}, let $B = B(\bm{N}, H)$ be the set given by \eqref{def:Box}. Given parameters $\tau, \epsilon$ and any constant $\mu>0$, write 
$$B_* =  \bigg\{\sum_{i=1}^n x_i\omega_i: x_i\in \left[-p^{-\tau}H,  p^{-\tau}H\right], 1\leq i\leq n\bigg\}.$$ 
and let
\begin{equation*}
    \Omega = \bigg\{\xi\in \F_q: \sum_{\substack{c\in B_*}} \bigg|\sum_{x\in B}\psi\bigg(\frac{c\xi}{x(x+c)}\bigg)\bigg|>\mu\cdot p^{-\epsilon}q^{-\tau}|B|^{2}\bigg\}.
\end{equation*}
Assume $0<\tau\leq 1$,
$0<\epsilon<\tau/2$, and that $p^\tau\leq H$. Then,
\begin{equation*}
    |\Omega| \lesssim_n q^{1+\tau + 16\epsilon/\tau} |B|^{-2}.
\end{equation*}
\end{lemma}
\begin{proof}
Note that for $c\in B_*$ and $b\in I = (0, p^{\tau/2})$, we have
\begin{align*}
   \bigg| &\sum_{x\in B}\psi\bigg(\frac{c\xi}{x(x+c)}\bigg) - \sum_{x\in B}\psi\bigg(\frac{c\xi}{(x+cb)(x+c(b+1))}\bigg)\bigg| \\ &\leq |B\setminus (B+cb)| + |(B+cb)\setminus B| \ll_n p^{-\tau/2}|B|.
\end{align*}
Since $\epsilon<\tau/2$, we have
$p^{\tau/2-\epsilon}\to\infty$ as $p\to\infty$. Hence, for
$p\geq p_0(n,\mu,\tau,\epsilon)$, the preceding translation
error may be absorbed into the defining lower bound for
$\Omega$. It follows that, for every $\xi\in\Omega$,
\begin{equation*}
    \sum_{\substack{c\in B_*}} \bigg|\sum_{x\in B}\sum_{b\in I}\psi\bigg(\frac{\xi}{c(c^{-1}x+b)(c^{-1}x+b+1)}\bigg)\bigg|\gg_n|B|^{2}q^{-\tau}p^{-\epsilon +\tau/2}.
\end{equation*}
Thus
\begin{equation}
\label{eqn:SpeclemLB}
    \sum_{\xi\in \Omega}\sum_{c\in B_*}\sum_{x\in B}\bigg|\sum_{b\in I}\psi\bigg(\frac{\xi}{c(c^{-1}x+b)(c^{-1}x+b+1)}\bigg)\bigg|\gg_n|\Omega||B|^{2}q^{-\tau}p^{-\epsilon+\tau/2}.
\end{equation}
For $\lambda_1$, $\lambda_2\in \F_q$, we define 
\begin{equation*}
    r(\lambda_1, \lambda_2) = |\{(x, c, \xi)\in B\times B_*\times \Omega: (\xi/c, x/c) = (\lambda_1, \lambda_2)\}|,
\end{equation*}
so that
\begin{equation}
\label{eqn:lem21stm}
    \sum_{\lambda_1, \lambda_2}r(\lambda_1, \lambda_2) = |\Omega||B||B_*|\approx_n q^{-\tau}|\Omega||B|^2.
\end{equation}
Further note that
\begin{equation*}
    \sum_{\lambda_1, \lambda_2}r(\lambda_1, \lambda_2)^2 = \bigg|\bigg\{(x_1, x_2, c_1, c_2, \xi_1, \xi_2)\in B^2\times B_*^2\times \Omega^2: \bigg(\frac{\xi_1}{c_1}, \frac{x_1}{c_1}\bigg) =\bigg(\frac{\xi_2}{c_2}, \frac{x_2}{c_2}\bigg) \bigg\}\bigg|
\end{equation*}
such that, by~\eqref{eqn:ECS}, Lemma~\ref{lem:KonMEB} and Lemma~\ref{lem:wrapped_box_energy}, we have
\begin{align}
\label{eqn:lem22ndm}
 \nonumber   \sum_{\lambda_1, \lambda_2}r(\lambda_1, \lambda_2)^2  &<|\Omega|E_{\times}(B, B_*)\\ 
 \nonumber &\leq |\Omega|E_{\times}(B)^{1/2}E_{\times}(B_*)^{1/2}\\
    &\lesssim_n |\Omega||B|^{2}q^{-\tau}.
\end{align}
We proceed to bound from above, the left-hand side of \eqref{eqn:SpeclemLB}, using a double use of H\"older's inequality, \eqref{eqn:lem21stm} and \eqref{eqn:lem22ndm}, as follows
\begin{align*}
    &\sum_{\lambda_1, \lambda_2}r(\lambda_1, \lambda_2)\bigg|\sum_{b\in I}\psi\bigg(\frac{\lambda_1}{(\lambda_2+b)(\lambda_2+b+1)}\bigg)\bigg|\\
    &\leq \bigg(\sum_{\lambda_1, \lambda_2}r(\lambda_1, \lambda_2)\bigg)^{1-\frac{1}{k}}\bigg(\sum_{\lambda_1, \lambda_2}r(\lambda_1, \lambda_2)^2\bigg)^{\frac{1}{2k}}\bigg(\sum_{\lambda_1, \lambda_2}\bigg|\sum_{b\in I}\psi\bigg(\frac{\lambda_1}{(\lambda_2+b)(\lambda_2+b+1)}\bigg)\bigg|^{2k}\bigg)^{\frac{1}{2k}}\\
    &\lesssim_n (|\Omega|q^{-\tau}|B|^2)^{1-\frac{1}{2k}}\cdot(*)^{\frac{1}{2k}},
\end{align*}
where $k$ is arbitrary and
\begin{align*}
   (*) &=  \sum_{\lambda_1, \lambda_2}\sum_{b_1, \dots, b_{2k}\in I^{2k}}\psi\bigg(\lambda_2\bigg(\frac{1}{(\lambda_1+b_1)(\lambda_1+b_1+1)}+ \cdots - \frac{1}{(\lambda_2+b_{2k})(\lambda_2+b_{2k}+1)}\bigg)\bigg)\\
   &=q\bigg|\bigg\{(b_1, \dots, b_{2k}, \lambda_1): \frac{1}{(\lambda_2+b_1)(\lambda_2+b_1+1)}+ \cdots - \frac{1}{(\lambda_2+b_{2k})(\lambda_2+b_{2k}+1)} = 0\bigg\}\bigg|\\
   &\ll q(p^{\tau k} + qp^{\frac{\tau k}{2}}).
\end{align*}
If $k>2n/\tau$, then $qp^{\tau k/2}\leq p^{\tau k}$. Indeed, since $q=p^n$, this inequality is equivalent to
$n\leq \tau k/2$. That is,
\begin{equation}
    (*)\lesssim qp^{\tau k}.
\end{equation}
Thus, returning to \eqref{eqn:SpeclemLB}, we get
$$
|\Omega||B|^2q^{-\tau}p^{-\epsilon+\tau/2}\lesssim_n (|\Omega|q^{-\tau}|B|^2)^{1-\frac{1}{2k}}\cdot(qp^{\tau k})^{\frac{1}{2k}},
$$
which further implies that
\begin{equation*}
    |\Omega| \lesssim_n |B|^{-2}q^{1+\tau}p^{2\epsilon k}.
\end{equation*}
We fix $k = 4n \lceil \tau^{-1}\rceil$, which gives the desired result.
\end{proof}
\begin{proof}[Proof of Theorem~\ref{thm:BKSOB}]
Let 
$$
W = \sum_{x\in B_1}\sum_{y\in B_2}\alpha(x)\beta(y) \psi(axy + bx^{-1}y^{-1}).
$$
By the triangle inequality, we get
$$
|W| \leq \sum_{x\in B_1}\left|\sum_{y\in B_2}\beta(y) \psi(axy + bx^{-1}y^{-1})\right|.
$$
Given arbitrary and small $\tau>0$, to be chosen later, we partition $B_1$ into $q^\tau$ subcubes $B_{1, \alpha}$, with $|B_{1, \alpha}| =q^{-\tau}H_1^n = q^{-\tau}|B_1|$. Thus, we have
$$
|W| \leq \sum_{\alpha=1}^{q^{\tau}}\sum_{x\in B_{1, \alpha}}\bigg|\sum_{y\in B_2}\beta(y) \psi(axy + bx^{-1}y^{-1})\bigg|.
$$
By the Cauchy-Schwarz inequality, we have
\begin{align*}
|W|^2 &\leq |B_1| \sum_{\alpha=1}^{q^{\tau}}\sum_{x\in B_{1, \alpha}}\bigg|\sum_{y\in B_2}\beta(y) \psi(axy + bx^{-1}y^{-1})\bigg|^2 \\
&= |B_1| \sum_{\alpha=1}^{q^{\tau}}\sum_{x\in B_{1, \alpha}}\sum_{y_1, y_2\in B_2}\beta(y_1)\overline{\beta}(y_2) \psi(ax(y_1-y_2) + bx^{-1}(y_1^{-1} - y_2^{-1}))\\
&\leq |B_1| \sum_{y_1, y_2\in B_2}\sum_{\alpha=1}^{q^{\tau}}\bigg|\sum_{x\in B_{1, \alpha}} \psi(ax(y_1-y_2) + bx^{-1}(y_1^{-1} - y_2^{-1}))\bigg|.
\end{align*}
Another application of the Cauchy-Schwarz inequality gives
\begin{align*}
   |W|^4 &\leq q^{\tau}|B_1|^{2}|B_2|^2\sum_{y_1, y_2\in B_2}\sum_{\alpha=1}^{q^{\tau}}\bigg|\sum_{x\in B_{1, \alpha}} \psi(ax(y_1-y_2) + bx^{-1}(y_1^{-1} - y_2^{-1}))\bigg|^2 \\
   &=q^{\tau}|B_1|^{2}|B_2|^2\sum_{y_1, y_2\in B_2}\sum_{\alpha=1}^{q^{\tau}}\sum_{x_1,x_2\in B_{1, \alpha}} \psi(a(x_1-x_2)(y_1-y_2) + b(x_1^{-1}-x_2^{-1})(y_1^{-1} - y_2^{-1})) \\ 
   &\leq q^{\tau}|B_1|^{2}|B_2|^2\sum_{\alpha=1}^{q^{\tau}}\sum_{x_1,x_2\in B_{1, \alpha}}\bigg|\sum_{y_1, y_2\in B_2} \psi(a(x_1-x_2)(y_1-y_2) + b(x_1^{-1}-x_2^{-1})(y_1^{-1} - y_2^{-1}))\bigg|\\
   &\leq q^{\tau}|B_1|^{2}|B_2|^2\sum_{\substack{x_1,x_2\in B_1\\ x_1-x_2\in B_*}}\bigg|\sum_{y_1, y_2\in B_2} \psi(a(x_1-x_2)(y_1-y_2) + b(x_1^{-1}-x_2^{-1})(y_1^{-1} - y_2^{-1}))\bigg|.
\end{align*}

For arbitrary $k$, by H\"{o}lder's inequality, we have
\begin{align*}
    |W|^{4k} \leq q^{\tau}|B_1|^{4k-2}|B_2|^{2k}\sum_{\substack{x_1,x_2\in B_1\\ x_1-x_2\in B_*}}&\sum_{y_1,\dots, y_{2k}\in B_2} \psi(a(x_1-x_2)(y_1-y_2+y_3\cdots-y_{2k}) \\ &+ b(x_1^{-1}-x_2^{-1})(y_1^{-1} - y_2^{-1}+y_3^{-1} \cdots -y_{2k}^{-1})).
\end{align*}
We further deduce that
\begin{align*}
    |W|^{4k} \leq q^{\tau}|B_1|^{4k-2}|B_2|^{2k}\sum_{y_1,\dots, y_{2k}\in B_2}\sum_{\substack{c\in B_*}} \bigg|\sum_{x\in B_1}\psi\bigg(b(y_1^{-1} - y_2^{-1}+y_3^{-1} \cdots -y_{2k}^{-1})\frac{c}{x(x+c)}\bigg)\bigg|.
\end{align*}
Now, for an arbitrary $\delta>0$, we assume, for a contradiction, that $|W|> p^{-\delta}|B_1||B_2|$. This gives
\begin{equation*}
    \sum_{y_1,\dots, y_{2k}\in B_2}\sum_{\substack{c\in B_*}} \bigg|\sum_{x\in B_1}\psi\bigg(b(y_1^{-1} - y_2^{-1}+y_3^{-1} \cdots -y_{2k}^{-1})\frac{c}{x(x+c)}\bigg)\bigg| > p^{-4\delta k}q^{ -\tau}|B_1|^2|B_2|^{2k}.
\end{equation*}

Write $\eta>0$ to be chosen later, take $\tau=\sqrt{\eta}$, and take $\delta=\eta/4k$. We apply Lemma~\ref{lem:PopPig} with $X=B_2^{2k}$,
$K=p^{-4\delta k}q^{-\tau}|B_1|^2|B_2|^{2k}$, 
$\mu=1/2$, and, for $\tilde{x}=(y_1,\dots,y_{2k})$, we take
\[
f(\tilde{x})
=
\sum_{c\in B_*}
\bigg|
\sum_{x\in B_1}
\psi\bigg(
b(y_1^{-1}-y_2^{-1}+y_3^{-1}\cdots-y_{2k}^{-1})
\frac{c}{x(x+c)}
\bigg)
\bigg|.
\]
Hence, there exists $Y\subset X$ such that
\begin{equation}
\label{eqn:YcardLB}
|Y|\gg p^{-\eta}|B_2|^{2k},
\end{equation}
and $f(\tilde{x})
\gg
p^{-\eta}q^{-\tau}|B_1|^2$ for every $\tilde{x}\in Y$.

Now let $\Omega$ be the set defined in Lemma~\ref{lem:LSpec2}, taking
$\epsilon=\eta$. Since $\eta<\tau/2$ for all sufficiently small $\eta$, Lemma~\ref{lem:LSpec2} yields
\[
|\Omega|
\lesssim_n
q^{1+\tau+16\eta/\tau}|B_1|^{-2}
=
q^{1+17\sqrt{\eta}}|B_1|^{-2}.
\]

We apply Corollary~\ref{cor:convconc}, taking $A=B_2^{-1}\cup(-B_2^{-1})$ and $S=b\Omega$. Since $|A|\leq 2|B_2|$, \eqref{eqn:YcardLB} implies $M(A)\gg p^{-\eta}|A|^{2k}$. Therefore, provided $k\geq k_0(n,\eta)$,
where $k_0(n,\eta)$ is the integer supplied by
Corollary~\ref{cor:convconc}, we obtain $|\Omega|
>
p^{-2\eta}|B_2|^2$. Consequently,
\[
|B_1|^2|B_2|^2
\lesssim_n
q^{1+17\sqrt{\eta}}p^{2\eta}.
\]
Since $\eta>0$ is arbitrary, we may choose it sufficiently small so that the right-hand side contradicts any choice of \eqref{eqn:B1B2LBC}, completing the proof.
\end{proof}

\begin{proof}[Proof of Corollary~\ref{thm:kfssrb}]
In the case that $B= B(\bm{N}, H)$ with $H\leq p^{1/2}$, the result follows immediately form Corollary~\ref{cor:convconc}. If $H>p^{1/2}$, then we may replace $B$ by $B_0 =B(\bm{N}, H_0)$ with $H_0 \approx p^{1/2}$ and apply Corollary~\ref{cor:convconc} again.
\end{proof}

\section{Proofs of Theorems~\ref{thm:digit-mixing} and \ref{thm:full-field-mixing}}
We present a number of auxiliary results, stated under assumptions and notations of Theorem~\ref{thm:digit-mixing}. Throughout, if $Z$ is an $\F_q$-valued random variable with distribution
$\mu_Z$, we write
\[
\widehat{\mu_Z}(a)
=
\sum_{z\in\F_q}\mu_Z(z)\overline{\psi_a(z)}
=
\mathbb E\overline{\psi_a}(Z).
\]

It follows immediately from Theorem~\ref{thm:BKSOB} that a single increment of the walk has small correlation with every
nontrivial additive character.

\begin{lemma}
\label{lem:one-step-small-bias}
Let $X$ and $Y$ be independent and uniformly distributed on
$B_1$ and $B_2$, respectively, and define $W=aXY+b(XY)^{-1}$ with $a\in\F_q$, $b\in\F_q^*$. Then, $\left|\mathbb E\psi(\xi W)\right|
\le \rho=\min\{1,C_\epsilon p^{-\delta}\}$ for every $\xi\in\F_q^*$.
\end{lemma}

\begin{proof}
For $\xi\ne0$,
\begin{align*}
\mathbb E\psi(\xi W)
&=
\frac{1}{|B_1||B_2|}
\sum_{x\in B_1}\sum_{y\in B_2}
\psi\left(\xi axy+\xi b x^{-1}y^{-1}\right).
\end{align*}
By Theorem~\ref{thm:BKSOB}, $\left|\mathbb E\psi(\xi W)\right|
\le C_\epsilon p^{-\delta}$. The trivial bound is at most $1$, giving the result.
\end{proof}


Let $W_i=aX_iY_i+b(X_iY_i)^{-1}$, where the pairs $(X_i,Y_i)$ are independent, with $X_i$ uniform on $B_1$ and $Y_i$ uniform on
$B_2$. Let $S_0$ be independent of the increments and define $S_k=S_0+W_1+\cdots+W_k$. 
We show the random walk forgets its initial state at an exponential rate.

\begin{lemma}
\label{lem:fourier-contraction}
For every $\xi\in\F_q^*$, $\left|\mathbb E\psi(\xi S_k)\right|
\le \rho^k$. More precisely,
\[
\mathbb E\psi(\xi S_k)
=
\mathbb E\psi(\xi S_0)
\prod_{i=1}^k\mathbb E\psi(\xi W_i).
\]
\end{lemma}

\begin{proof}
By independence,
\begin{align*}
\mathbb E\psi(\xi S_k)
&=
\mathbb E\psi\left(
\xi S_0+\xi W_1+\cdots+\xi W_k
\right)\\
&=
\mathbb E\psi(\xi S_0)
\prod_{i=1}^k\mathbb E\psi(\xi W_i).
\end{align*}
Since $\left|\mathbb E\psi(\xi S_0)\right|\le1$ and, by Lemma~\ref{lem:one-step-small-bias}, $\left|\mathbb E\psi(\xi W_i)\right|\le\rho$
for every $i$, the result follows.
\end{proof}

We recall, (see e.g. \cite[Theorem~2.24]{LN}), that every $\F_p$-linear map from $\F_q$ to $\F_p$ can be represented
using the trace map $Tr(x)$.

\begin{lemma}
\label{lem:trace-representation}
For every $\F_p$-linear map $L:\F_q\longrightarrow\F_p$, there exists a unique $\theta\in\F_q$ such that $L(z)=Tr(\theta z)$ for all $z\in\F_q$. Moreover, $L$ is nonzero if and only if
$\theta\ne0$.
\end{lemma}



\begin{lemma}
\label{lem:projection-fourier-decay}
Let $L:\F_q\to\F_p$ be a nonzero $\F_p$-linear map. Then,
for every $a\in\F_p^*$, $\left|
\mathbb E e_p\bigl(a L(S_k)\bigr)
\right|
\le\rho^k$.
\end{lemma}

\begin{proof}
By Lemma~\ref{lem:trace-representation}, there exists $\theta\in\F_q^*$ such that $L(z)= Tr(\theta z)$. For $a\in\F_p^*$, $a L(z)= Tr(a\theta z)$. The map $z\longmapsto e_p\bigl( Tr(a\theta z)\bigr)$ is a nontrivial additive character of $\F_q$, since $a\theta\ne0$. Lemma~\ref{lem:fourier-contraction} therefore gives  $\left| \mathbb E e_p\bigl(a L(S_k)\bigr) \right| \le\rho^k$.
\end{proof}


We now relate Fourier decay to several notions of equidistribution for probability distributions on finite vector spaces.

\begin{lemma}
\label{lem:entropy-chi-square}
For probability distributions $P$ and $Q$ on a finite set,
\[
D(P\Vert Q)
\le
\log\left(1+\chi^2(P\Vert Q)\right)
\]
and
\[
\|P-Q\|_{\mathrm{TV}}
\le
\frac12\sqrt{\chi^2(P\Vert Q)}.
\]
\end{lemma}

\begin{proof}
The first inequality follows from the monotonicity of R\'enyi
divergence in its order~\cite[Theorem~3]{vEH}. Indeed, $D(P\Vert Q)
\le
D_2(P\Vert Q)$,
where
\[
D_2(P\Vert Q)
=
\log\sum_x\frac{P(x)^2}{Q(x)}
=
\log\left(1+\chi^2(P\Vert Q)\right).
\]

For the second inequality, Cauchy--Schwarz gives
\begin{align*}
\sum_x|P(x)-Q(x)|
&=
\sum_x
\frac{|P(x)-Q(x)|}{\sqrt{Q(x)}}\sqrt{Q(x)}\\
&\le
\left(
\sum_x\frac{(P(x)-Q(x))^2}{Q(x)}
\right)^{1/2}
\left(\sum_xQ(x)\right)^{1/2}\\
&=
\sqrt{\chi^2(P\Vert Q)}.
\end{align*}
\end{proof}

\begin{lemma}
\label{lem:chi-square-fourier}
If $\mu$ is a probability distribution on $\F_q$, then
\[
\chi^2(\mu\Vert U_q)
=
\sum_{\xi\in\F_q^*}|\widehat\mu(\xi)|^2.
\]
\end{lemma}

\begin{proof}
By definition,
\[
\chi^2(\mu\Vert U_q)
=
q\sum_{z\in\F_q}
\left(\mu(z)-\frac1q\right)^2.
\]
Let $f(z)=\mu(z)-1/q$. Since
\[
\sum_z f(z)=0,
\]
we have $\widehat f(0)=0$. For
$\xi\ne0$,
\[
q\widehat f(\xi)
=
\sum_z\mu(z)\overline{\psi(\xi z)}
=
\widehat\mu(\xi),
\]
because the Fourier transform of the uniform distribution vanishes at every
nonzero frequency. Parseval's identity gives
\[
\sum_z|f(z)|^2
=
q\sum_{\xi\in\F_q}|\widehat f(\xi)|^2.
\]
Multiplying by $q$ and using $\widehat f(0)=0$, we obtain
\[
\chi^2(\mu\Vert U_q)
=
q^2\sum_{\xi\in\F_q^*}|\widehat f(\xi)|^2
=
\sum_{\xi\in\F_q^*}|\widehat\mu(\xi)|^2,
\]
as required.
\end{proof}

\begin{proof}[Proof of Theorem~\ref{thm:digit-mixing}]
Fix a nonzero $\F_p$-linear map $L$, and let $P_{k,L}$
denote the distribution of $L(S_k)$ on $\F_p$. For every
$s\in\F_p^*$, Lemma~\ref{lem:projection-fourier-decay} gives
\[
|\widehat P_{k,L}(s)|
=
\left|
\mathbb E e_p(sL(S_k))
\right|
\le\rho^k.
\]
Lemma~\ref{lem:chi-square-fourier} therefore gives
\begin{align*}
\chi^2(P_{k,L}\Vert U_p)
=
\sum_{s\in\F_p^*}
|\widehat P_{k,L}(s)|^2
\le
\sum_{s\in\F_p^*}\rho^{2k}
=
(p-1)\rho^{2k}.
\end{align*}
This proves \eqref{eq:projection-chi-square}. Applying Lemma~\ref{lem:entropy-chi-square} gives
\[
\|P_{k,L}-U_p\|_{\mathrm{TV}}
\le
\frac12
\sqrt{\chi^2(P_{k,L}\Vert U_p)}
\le
\frac12\sqrt{p-1}\,\rho^k,
\]
which proves \eqref{eq:projection-TV}. Since $U_p(t)=1/p$, one has $D(P_{k,L}\Vert U_p)
=
\log p-H(P_{k,L})$.
Therefore
\begin{align*}
\log p-H(P_{k,L})
&=
D(P_{k,L}\Vert U_p)\\
&\le
\log\left(
1+\chi^2(P_{k,L}\Vert U_p)
\right)\\
&\le
\log\left(1+(p-1)\rho^{2k}\right).
\end{align*}
Rearranging proves \eqref{eq:projection-entropy}. Finally, each coordinate map $\pi_j:\F_q\to\F_p$ is a nonzero $\F_p$-linear map. Hence
\[
H(\pi_j(S_k))
\ge
\log p-
\log\left(1+(p-1)\rho^{2k}\right)
\]
for every $j$. Summing over $j=1,\dots,n$ yields
\[
H_\Omega(S_k)
\ge
n\log p-
n\log\left(1+(p-1)\rho^{2k}\right).
\]
\end{proof}

\begin{proof}[Proof of Theorem~\ref{thm:full-field-mixing}]
Let $\mu_k$ denote the distribution of $S_k$. By
Lemma~\ref{lem:fourier-contraction}, $|\widehat\mu_k(\xi)|\le\rho^k$ for every $\xi\in\F_q^*$. Hence
\begin{align*}
\chi^2(\mu_k\Vert U_q)
&=
\sum_{\xi\in\F_q^*}
|\widehat\mu_k(\xi)|^2
\le
(q-1)\rho^{2k},
\end{align*}
where we used Lemma~\ref{lem:chi-square-fourier}. This proves
\eqref{eq:full-chi-square}.

Lemma~\ref{lem:entropy-chi-square} now gives
\[
\|\mu_k-U_q\|_{\mathrm{TV}}
\le
\frac12\sqrt{q-1}\,\rho^k.
\]
Also, $D(\mu_k\Vert U_q)
=
\log q-H(\mu_k)$,
so
\begin{align*}
\log q-H(\mu_k)
&\le
\log\left(1+\chi^2(\mu_k\Vert U_q)\right)\\
&\le
\log\left(1+(q-1)\rho^{2k}\right).
\end{align*}
This proves \eqref{eq:full-entropy}.
\end{proof}


\begin{thebibliography}{}
\bibitem{Bak} R. C. Baker,
`Kloosterman sums with prime variable',
{\it Acta Arith.\/}, {\bf 156(4)} (2012), 351--372.

\bibitem{Bou} J.~Bourgain,
`More on the sum-product phenomenon in prime fields and its applications',
{\it Int. J. Number Theory\/}, {\bf 1(1)} (2005), 1--32.

\bibitem{BouGar2} J.~Bourgain and M.~Z.~Garaev,
`Sumsets of reciprocals in prime fields and multilinear Kloosterman sums',
{\it Izv. Ross. Akad. Nauk Ser. Mat.\/}, {\bf 78(4)} (2014), 19--72; translation in {\it Izv. Mat.\/}, {\bf 78} (2014), 656--707.

\bibitem{Chang} M.-C.~Chang,
`An estimate of incomplete mixed character sums', in:
{\it An irregular mind\/}, Bolyai Society Mathematical studies, 21 (Springer, Berlin 2010), 242--250.

\bibitem{Kon} S.V. Konyagin, {\it Estimates of character sums in finite fields}, {\it Math. Notes\/}, {88} (2010), 503--515.

\bibitem{KonShk} S.V. Konyagin, I.D. Shkredov, {\it New results on sums and products in $\R$}, {\it Proc. Steklov Inst. Math.\/}, {294} (2016), 78--88.

\bibitem{LN}
R.~Lidl and H.~Niederreiter,
\textit{Finite Fields},
2nd ed.,
Encyclopedia of Mathematics and its Applications, Vol.~20,
Cambridge Univ. Press, Cambridge, 1997.

\bibitem{TaoVu}
T. Tao and V. Vu, `Additive combinatorics', {\it Cambridge Univ. Press\/}, 2006.

\bibitem{vEH}
T.~van Erven and P.~Harremo\"es,
`R\'enyi divergence and Kullback--Leibler divergence',
{\it IEEE Trans. Inform. Theory\/},
{\bf 60(7)} (2014), 3797--3820.
 \end{thebibliography}
\end{document}